\documentclass[11pt]{amsart}
\usepackage{amsmath,amssymb}
\usepackage{verbatim}
\usepackage{color}
\usepackage{hyperref}

\begin{document}
\newcommand{\chp}{\mathbf 1_{\mathcal P}}

\newtheorem{thm}{Theorem}[section]
\newtheorem{theorem}[thm]{Theorem}
\newtheorem{lem}[thm]{Lemma}
\newtheorem{lemma}[thm]{Lemma}
\newtheorem{prop}[thm]{Proposition}
\newtheorem{proposition}[thm]{Proposition}
\newtheorem{cor}[thm]{Corollary}
\newtheorem{defn}[thm]{Definition}
\newtheorem*{remark}{Remark}
\newtheorem{conj}[thm]{Conjecture}

\numberwithin{equation}{section}

\newcommand{\Z}{{\mathbb Z}}
\newcommand{\Q}{{\mathbb Q}}
\newcommand{\R}{{\mathbb R}}
\newcommand{\C}{{\mathbb C}}
\newcommand{\N}{{\mathbb N}}
\newcommand{\FF}{{\mathbb F}}
\newcommand{\fq}{\mathbb{F}_q}
\newcommand{\rmk}[1]{\footnote{{\bf Comment:} #1}}

\newcommand{\bfA}{{\boldsymbol{A}}}
\newcommand{\bfY}{{\boldsymbol{Y}}}
\newcommand{\bfX}{{\boldsymbol{X}}}
\newcommand{\bfZ}{{\boldsymbol{Z}}}
\newcommand{\bfa}{{\boldsymbol{a}}}
\newcommand{\bfy}{{\boldsymbol{y}}}
\newcommand{\bfx}{{\boldsymbol{x}}}
\newcommand{\bfz}{{\boldsymbol{z}}}
\newcommand{\F}{\mathcal{F}}
\newcommand{\Gal}{\mathrm{Gal}}
\newcommand{\Fr}{\mathrm{Fr}}
\newcommand{\Hom}{\mathrm{Hom}}
\newcommand{\GL}{\mathrm{GL}}

\renewcommand{\mod}{\;\operatorname{mod}}
\newcommand{\ord}{\operatorname{ord}}
\newcommand{\TT}{\mathbb{T}}

\renewcommand{\i}{{\mathrm{i}}}
\renewcommand{\d}{{\mathrm{d}}}
\renewcommand{\^}{\widehat}
\newcommand{\HH}{\mathbb H}
\newcommand{\Vol}{\operatorname{vol}}
\newcommand{\area}{\operatorname{area}}
\newcommand{\tr}{\operatorname{tr}}
\newcommand{\norm}{\mathcal N}
\newcommand{\intinf}{\int_{-\infty}^\infty}
\newcommand{\ave}[1]{\left\langle#1\right\rangle} 
\newcommand{\Var}{\operatorname{Var}}
\newcommand{\Prob}{\operatorname{Prob}}
\newcommand{\sym}{\operatorname{Sym}}
\newcommand{\disc}{\operatorname{disc}}
\newcommand{\CA}{{\mathcal C}_A}
\newcommand{\cond}{\operatorname{cond}} 
\newcommand{\lcm}{\operatorname{lcm}}
\newcommand{\Kl}{\operatorname{Kl}} 
\newcommand{\leg}[2]{\left( \frac{#1}{#2} \right)}  
\newcommand{\Li}{\operatorname{Li}}

\newcommand{\sumstar}{\sideset \and^{*} \to \sum}

\newcommand{\LL}{\mathcal L} 
\newcommand{\sumf}{\sum^\flat}
\newcommand{\Hgev}{\mathcal H_{2g+2,q}}
\newcommand{\USp}{\operatorname{USp}}
\newcommand{\conv}{*}
\newcommand{\dist} {\operatorname{dist}}
\newcommand{\CF}{c_0} 
\newcommand{\kerp}{\mathcal K}

\newcommand{\Cov}{\operatorname{cov}}
\newcommand{\Sym}{\operatorname{Sym}}

\newcommand{\ES}{\mathcal S} 
\newcommand{\EN}{\mathcal N} 
\newcommand{\EM}{\mathcal M} 
\newcommand{\Sc}{\operatorname{Sc}} 
\newcommand{\Ht}{\operatorname{Ht}}

\newcommand{\E}{\operatorname{E}} 
\newcommand{\sign}{\operatorname{sign}} 

\newcommand{\divid}{d} 

\title[Divisor problems over function fields]
{Shifted convolution and the Titchmarsh divisor problem over
$\fq[t]$}
\author{J.~C.~Andrade, L.~Bary-Soroker and Z.~Rudnick}
\address{Institut des Hautes \'{E}tudes Scientifiques (IH\'{E}S),  Le Bois-Marie 35, route de Chartres, Bures-sur-Yvette, 91440, France}
\email{j.c.andrade@ihes.fr}

\address{Raymond and Beverly Sackler School of Mathematical Sciences, Tel Aviv University, Tel Aviv 69978, Israel}
\email{barylior@post.tau.ac.il}

\address{Raymond and Beverly Sackler School of Mathematical Sciences, Tel Aviv University, Tel Aviv 69978, Israel}
\email{rudnick@post.tau.ac.il}

\thanks{JCA is supported by an IH\'{E}S Postdoctoral Fellowship and an EPSRC William Hodge Fellowship. \\
The research leading to these results has received funding from the
European Research Council under the European Union's Seventh
Framework Programme (FP7/2007-2013) / ERC grant agreement
n$^{\text{o}}$ 320755, and from the Israel Science Foundation (grant
No. 925/14).}


\begin{abstract}
In this paper we solve a function field analogue of classical
problems in analytic number theory, concerning the
auto--correlations of divisor functions, in the limit of a large
finite field.
\end{abstract}
\date{\today}

\maketitle

\section{Introduction}

The goal of this note is to study a function-field analogue of
classical problems in analytic number theory, concerning the
auto-correlations of divisor functions. First we review the problems
over the integers $\mathbb{Z}$ and then we proceed to investigate
the same problems over the rational function field
$\mathbb{F}_{q}(t)$.

\subsection{The additive divisor problem and over $\Z$}

Let $\divid_k(n)$ be the number of representations of $n$ as a
product of $k$ positive integers ($\divid_2$ is the standard
divisor function). Several authors have studied the \textbf{additive
divisor problem} (other names are ``shifted divisor" and ``shifted
convolution"), which is to get bounds, or asymptotics, for the sum

\begin{equation}
\label{eq:main} D_k(x;h):=\sum_{n\leq x}\divid_k(n)\divid_k(n+h),
\end{equation}
where $h\neq 0$ is fixed for this discussion.

The case $k=2$ (the ordinary divisor function) has a long history:
Ingham \cite{Ingham} computed the leading term, and Estermann
\cite{Estermann} gave an asymptotic expansion
\begin{equation}
\label{eq:Estermann} \sum_{n\leq x} \divid_2(n)\divid_2(n+h) = x
P_2(\log x;h) +O(x^{11/12 } ( \log x)^3),
\end{equation}
where
\begin{equation}\label{Quadratic Z}
P_2(u;h) = \frac 1{\zeta(2)} \sigma_{-1}(h) u^2 +a_1(h) u+a_2(h),
\end{equation}
with
\begin{equation}
\sigma_{w}(h)=\sum_{d\mid h}d^{w},
\end{equation}
and $a_1(h)$, $a_2(h)$ are very complicated coefficients.

The size of the remainder term has great importance in applications
for various problems in analytic number theory, in particular the
dependence on $h$. See \cite{DI, HB}  for an improvement of the
remainder term.

The higher divisor problem $k\geq3$ is also of importance, in
particular in relation to computing the moments of the Riemann
$\zeta$-function on the critical line, see \cite{CG, Ivic}. It is
conjectured that
\begin{equation}\label{conj CG}
D_k(x;h)\sim xP_{2(k-1)}(\log x;h) \qquad \text{as
$x\rightarrow\infty$},
\end{equation}
where  $P_{2(k-1)}(u;h)$ is a polynomial in $u$ of degree $2(k-1)$,
whose coefficients depend on $h$ (and $k$). We can get good upper
bounds on the additive divisor problem from results in sieve theory
on sums of multiplicative functions evaluated at polynomials, for
instance as those by Nair and Tenenbaum \cite{NaTe}. The conclusion
is that for $h\neq0$
\begin{equation}
\label{eq-upperbounds} \sum_{n\leq X}d_{k}(n)d_{k}(n+h)\ll X(\log
X)^{2(k-1)},
\end{equation}
and we believe this is the right order of magnitude. But even a
conjectural description of the polynomials $P_{2(k-1)}(u;h)$ is
difficult to obtain, see \cite{CG, Ivic}, see
\S~\ref{sec:comparison}.

A variant of the problem about the auto-correlation of the divisor
function, is to determine an asymptotic for the more general
sum given by
\begin{equation}\label{eq:generaldivisor}
D_{k,r}(x;h):=\sum_{n\leq x}d_{k}(n)d_{r}(n+h).
\end{equation}
Asymptotics are known for the case $(k,r)=(k,2)$ for any positive
integer $k\geq 2$:
 Linnik \cite{Lin} showed
\begin{align}
\label{eq:Lin}
D_{k,2}(x;1)&=\sum_{n\leq x}d_{k}(n)d_{2}(n+1)\nonumber\\
&=\frac{1}{(k-1)!}\prod_{p}\left(1-\frac{1}{p}+\frac{1}{p}\left(1-\frac{1}{p}\right)^{k-1}\right)x(\log x)^{k}\\
&\qquad {} +O\left(x(\log x)^{k-1}(\log\log
x)^{k^{4}}\right).\nonumber
\end{align}

Motohashi  \cite{Mot1, Mot2, Mot3}  gave an asymptotic expansion
\begin{equation}
\label{eq:Mot2} D_{k;2}(x,h)=x\sum_{j=0}^{k}f_{k,j}(h)(\log
x)^{j}+O(x(\log x)^{\varepsilon-1}),
\end{equation}
for all $\varepsilon>0$ where the coefficients $f_{k,j}(h)$ can in
principle  be explicitly computed. For an improvement in the
$O$--term see \cite{FoTe}.

\subsection{The Titchmarsh divisor problem over $\Z$}
A different problem involving the mean value of the divisor function
is the \textbf{Titchmarsh divisor problem}. The problem is to
understand the average behaviour of the number of divisors of a
shifted prime, that is the asymptotics of the sum over primes
\begin{equation}
\label{eq:Tit} \sum_{p\leq x}d_{2}(p+a)
\end{equation}
where $a\neq0$ is a fixed integer, and $x\rightarrow\infty$.
Assuming GRH, Titchmarsh \cite{Tit} showed in 1931 that
\begin{equation}
\label{eq:Tit1} \sum_{p\leq x}d_{2}(p+a)\sim C_{1}x
\end{equation}
with
\begin{equation}
\label{eq:Tit2} C_{1}=\frac{\zeta(2)\zeta(3)}{\zeta(6)}\prod_{p\mid
a}\left(1-\frac{p}{p^{2}-p+1}\right)
\end{equation}
and this was proved unconditionally by Linnik \cite{Lin} in 1963.

Fouvry \cite{Fou} and Bombieri, Friedlander and Iwaniec \cite{BFI}
gave a secondary term
\begin{equation}
\label{eq:Tit3} \sum_{p\leq
x}d_{2}(p+a)=C_{1}x+C_{2}\Li(x)+O\left(\frac{x}{(\log x)^{A}}\right),
\end{equation}
for all $A>1$ and
\begin{equation}
\label{eq:Tit4} C_{2}=C_{1}\left(\gamma-\sum_{p}\frac{\log
p}{p^{2}-p+1}+\sum_{p\mid a}\frac{p^{2}\log
p}{(p-1)(p^{2}-p+1)}\right)
\end{equation}
with $\gamma$ being the Euler-Mascheroni constant and $\Li(x)$ the
logarithmic integral function.

In the following sections we study the additive divisor problem and
the Titchmarsh divisor problem over $\mathbb{F}_{q}[t]$, obtaining
definitive analogues of the conjectures described above.

\subsection{The additive divisor problem over $\fq[t]$}

\label{sec: additive divisor problem} We denote by $\EM_n$ the set
of monic polynomials in $\fq[t]$ of degree $n$. Note that $\#\EM_n =
q^n$.

The  divisor function  $\divid_k(f)$ is the number of ways to write
a monic polynomial $f$ as a product of $k$ monic polynomials:
\begin{equation}
\divid_k(f) = \#\{(a_1,\dots,a_k) , f=a_1\cdot a_2\cdots a_k\},
\end{equation}
where it is allowed to have $a_{i}=1$.

The mean value of $\divid_k(f)$ has an exact formula (see
Lemma~\ref{lem:mean value of tau}):
\begin{equation}\label{eq:mvdiv_k}
\frac{1}{q^n}\sum_{f\in \EM_n} \divid_{k}(f) = \binom{n+k-1}{k-1}.
\end{equation}
Note that $\binom{n+k-1}{k-1}$ is a polynomial in $n$ of degree
$k-1$ and leading coefficient $1/(k-1)!$. Our first goal is to study
the auto-correlation of $\divid_k$
 in the limit $q\to \infty$. We show:

\begin{theorem}\label{main thm}
Fix $n>1$. Then
\begin{equation}
\label{eq:main thm} \frac 1{q^n} \sum_{f\in \EM_n}
\divid_k(f)\divid_k(f+h) = \binom{n+k-1}{k-1}^2  +
O\left(q^{-\frac12}\right),
\end{equation}
uniformly for all $0\neq h \in \FF_q[t]$ of degree $\deg(h)<n$, as
$q\rightarrow\infty$.
\end{theorem}

In light of \eqref{eq:mvdiv_k}, Theorem~\ref{main thm} may be
interpreted as the statement that $d_k(f)$ and $d_k(f+h)$ become
independent in the limit $q\to \infty$ as long as $\deg(h)<n$.

To compare with conjecture \eqref{conj CG} over $\Z$ we note that
$\binom{n+k-1}{k-1}^2$ is a polynomial in $n$ of degree $2(k-1)$
with leading coefficient $1/[(k-1)!]^2$, in agreements with the
conjecture, see \S~\ref{Sec:compatibility}.

\noindent{\bf The case $h=0$:} As an aside, we note that the case
$h=0$ is of course dramatically different, indeed one can show that
\begin{equation}
\lim_{q\to \infty} \frac 1{q^n} \sum_{f\in \EM_n} \divid_k(f)^2  =
\binom{n+k^2-1}{k^2-1}
\end{equation}
is a polynomial of degree $k^2-1$ in $n$, rather than degree
$2(k-1)$ for nonzero shifts.



Our method in fact gives the more general result:
\begin{theorem}
\label{thm:general case} Let $\mathbf{k}= (k_1, \ldots, k_s)$ be a
tuple of positive integers and $\mathbf{h}=(h_1, \ldots, h_s)$ a
tuple of distinct polynomials in $\fq[t]$. We let
\[
D_{\mathbf{k}}(n;\mathbf{h})=\sum_{f\in\mathcal{M}_{n}}d_{k_1}(f+h_1)\cdots
d_{k_{s}}(f+h_{s}).
\]
Then, for fixed $n>1$,
\[
 \frac{1}{q^{n}}D_{\mathbf{k}}(n;\mathbf{h}) = \prod_{i=1}^s \binom{n+k_i-1}{k_i-1} +O\left(q^{-\frac{1}{2}}\right),
\]
uniformly on all tuples $\mathbf{h}=(h_{1},h_{2},\ldots, h_{s})$ of
distinct polynomials in $\mathbb{F}_{q}[t]$ of degrees
$\deg(h_{i})<n$ as $q\to \infty$.
\end{theorem}

In particular for $\mathbf{k}=(2,k)$ we get
\begin{align}
\lim_{q\rightarrow\infty}\frac{1}{q^{n}}D_{2,k}(n;h)&=(n+1)\binom{n+k-1}{k-1}\nonumber \\
&=\frac{1}{(k-1)!}\left(n^{k}+\frac{k^{2}-k+2}{2}n^{k-1}+\cdots\right),
\end{align}
in agreement with \eqref{eq:Lin}.


\subsection{The Titchmarsh divisor problem over $\fq[t]$}
Let $\mathcal{P}_{n}$ be the set of monic irreducible polynomials in
$\fq[t]$ of degree $n$. By the Prime Polynomial Theorem we have
\[
\pi_q(n):=\#\mathcal{P}_n=\frac{q^n}{n} +
O\left(\frac{q^{n/2}}{n}\right).
\]
Our next result is a solution of the  Titchmarsh divisor problem
over $\mathbb{F}_{q}[t]$ in the limit of large finite field.
\begin{theorem}
\label{TitDiv for P} Fix $n>1$.  Then
\begin{equation}
\label{TitDiv eq}
\frac{1}{\pi_q(n)}\sum_{P\in\mathcal{P}_{n}}d_{k}(P+\alpha)=\binom{n+k-1}{k-1}+O\left(q^{-\tfrac{1}{2}}\right)
\end{equation}
uniformly over all $0\neq\alpha\in \fq[t]$ of degree $\deg(\alpha)<n$.
\end{theorem}

For the standard divisor function ($k=2$) we find
\begin{equation}
\label{eq:Tit d2}
\sum_{P\in\mathcal{P}_{n}}d_{2}(P+\alpha)=q^{n}+\frac{q^{n}}{n}+O\left(q^{n-\tfrac{1}{2}}\right),
\end{equation}
which is analogous to \eqref{eq:Tit3} under the correspondence
$x\leftrightarrow q^{n}$, $\log x\leftrightarrow n$.

\subsection{Independence of cycle structure of shifted polynomials}
We conclude the introduction with a discussion on the connection
between shifted polynomials and random permutations and state a 
result that lies behind the results stated above.

The cycle structure of a permutation $\sigma$ of $n$ letters is the
partition $\lambda(\sigma) = (\lambda_1,\dots, \lambda_n)$ of $n$ if
in the decomposition of $\sigma$ as a product of disjoint cycles,
there are
 $\lambda_j$ cycles of length $j$. Note that $\lambda(\sigma)$ is a partition of $n$ is the sense that $\lambda_j\geq 0$ and $\sum_{j} j \lambda_j = n$. For example, $\lambda_1$ is the number of fixed points of $\sigma$ and $\lambda_n=1$ if and only if $\sigma$ is an $n$-cycle.

For each partition $\lambda \vdash n$, the probability that a random
permutation on $n$ letters has cycle structure $\sigma$ is given by
Cauchy's formula \cite[Chapter 1]{ABT}:
\begin{equation}\label{def p}
p(\lambda) = \frac{\#\{\sigma\in S_n: \lambda(\sigma)=
\lambda\}}{\#S_n } = \prod_{j=1}^n \frac 1{j^{\lambda_j} \cdot
\lambda_j!}.
\end{equation}

For $f\in \fq[t]$ of positive degree $n$, we say its cycle structure
is $\lambda(f) = (\lambda_1,\dots, \lambda_n)$ if in the prime
decomposition $f=\prod_j P_j$ (we allow repetition), we have $\#\{i:
\deg(P_i)=j\} = \lambda_j$. Thus we get a partition of $n$. In
analogy with permutation, $\lambda_1(f)$ is the number of roots of
$f$ in $\fq$ (with multiplicity) and $f$ is irreducible if and only
if $\lambda_n(f) = 1$.

For a partition $\lambda\vdash n$, we let $\chi_{\lambda}$ be the
characteristic function of $f\in \EM_n$ of cycle structure
$\lambda$:
\begin{equation}
\chi_\lambda(f)=\begin{cases} 1, & \lambda(f) = \lambda\\ 0,
&\mbox{otherwise}.\end{cases}
\end{equation}
The Prime Polynomial Theorem gives the mean values of
$\chi_\lambda$:
\begin{equation}\label{number of pols with cycle}
\frac{1}{q^n} \sum_{f\in \EM_n} \chi_{\lambda} (f) = p(\lambda) +
O\left(q^{-1}\right)
\end{equation}
as $q\to \infty$ (see Lemma~\ref{lem:number of pols with cycle}). We
prove independence of cycle structure of shifted polynomials:

\begin{theorem}\label{cycle indep thm}
For fixed positive integers $n$ and $s$ we have
\[
\frac{1}{q^n}\sum_{f\in \EM_n} \chi_{\lambda_1}(f+h_1) \cdots
\chi_{\lambda_s}(f+h_s) = p(\lambda_1) \cdots p(\lambda_s)
+O\left(q^{-\frac{1}{2}}\right),
\]
uniformly for all $h_1, \ldots, h_s$ distinct polynomials in $\fq[t]$ of degrees
$\deg(h_i)<n$ and on all partitions $\lambda_1, \ldots,
\lambda_s\vdash n$ as $q\to \infty$.
\end{theorem}
\begin{remark}
In this theorem $\lambda_{1},\cdots,\lambda_{s}$ are partitions of $n$ and are not the same as the $\lambda_{1},\cdots,\lambda_{n}$ that appears on the definition of $\lambda(f)$ or $\lambda(\sigma)$ where in that case the $\lambda_{i}$'s are the number of parts of length $i$.
\end{remark}


We note that the statistic of Theorem~\ref{cycle indep thm} is
induced from the statistics of cycles structure of tuples of
elements in the direct product $S_n^s$ of $s$ copies of the
symmetric group on $n$ letters $S_n$. This plays a role in the
proof, where we use that a certain Galois group is $S_n^s$
\cite{BS}, and we derive the statistic from an explicit Chebotarev
theorem. Since we have not found the exact formulation that we need
in the literature, we provide a proof in the Appendix.



\section{Mean values}\label{sec:prelim}
For the reader's convenience, we prove in this section 
some results for which we did not find a good reference. 
We  define the \textbf{norm} of a nonzero polynomial
$f\in\fq[t]$ to be $|f|=q^{\mathrm{deg}(f)}$ and set $|0|=0$.

We start by proving \eqref{number of pols with cycle}:

\begin{lem}\label{lem:number of pols with cycle}
If $\lambda \vdash n$ is a partition of $n$ and $n$ is a fixed
number then
\begin{equation}\label{number of pols with cycle2}
\frac 1{q^n} \#\{f\in \EM_n: \lambda(f) = \lambda\}  =
p(\lambda)(1+O(q^{-1}))
\end{equation}
as $q\rightarrow\infty$.
\end{lem}

\begin{proof}
To see this, note that to get a monic polynomial with cycle
structure $\lambda$, we pick any $\lambda_1$ primes of degree $1$,
$\lambda_2$ primes of degree $2$, (irrespective of the choice of
ordering), and multiply them together. Thus

\begin{equation}
 \#\{f\in \EM_n: \lambda(f) = \lambda\}  = \prod_{j=1}^n
 \frac{\pi_A(j)^{\lambda_j}}{\lambda_j!}\Big(1+O(\frac 1q) \Big)
 \end{equation}
where $\pi_{A}(j)$ is the number of primes of degree $j$ in
$A=\fq[t]$. By the Prime Polynomial Theorem,  $\pi_{A}(j) = \frac
{q^j}{j} + O(\frac{q^{j/2}}{j})$ whenever $j\geq 2$ and $\pi_A(1) =
q$. Hence $\pi_A(j) = \frac{q^j}{j}+O(\frac{q^{j-1}}{j})$. So
 \begin{equation}
 \begin{split}
\#\{f\in \EM_n: \lambda(f) = \lambda\}
    &= \prod_{j=1}^n \frac1{\lambda_j!} \left(\frac{q^j}{j} + O\left(\frac{q^{j-1}}{j}\right)\right)^{\lambda_j}\\
    & = q^{\sum
j\lambda_j} \prod_{j=1}^n\frac 1{j^{\lambda_j} \cdot \lambda_j!}
(1+O(q^{-1}))
\end{split}
\end{equation}
which by \eqref{def p} gives \eqref{number of pols with cycle2}.
\end{proof}
Next we prove \eqref{eq:mvdiv_k}:
\begin{lemma}\label{lem:mean value of tau}
The mean value of $\divid_k(f)$ is
\begin{equation}\label{eq:mean value of tau}
\frac 1{q^n} \sum_{f\in \EM_n} \divid_k(f) =\binom{n+k-1}{k-1}.
\end{equation}
\end{lemma}

\begin{proof}
The generating function for $\divid_k(f)$ is the $k$-th power of the
zeta function associated to the polynomial ring $\fq[t]$
\begin{equation}\label{gen fn od div}
Z(u)^k= \sum_{f\;{\rm monic}} \divid_k(f)u^{\deg f} =
\sum_{n=0}^{\infty} \sum_{f\in \EM_n} d_k(f) u^n.
\end{equation}
Here
\begin{equation}
Z(u) =  \sum_{f\;{\rm monic}} u^{\deg f} = \sum_{n=0}^{\infty} q^n
u^n = \frac 1{1-qu} .
\end{equation}
Using the Taylor expansion
\begin{equation}
\frac 1{(1-x)^k} = \sum_{n=0}^\infty \binom{n+k-1}{k-1} x^n
\end{equation}
and comparing the  coefficients of $u^n$ in \eqref{gen fn od div}
give
\begin{equation}
q^n \binom{n+k-1}{k-1} = \sum_{f\in \EM_n} \divid_k(f),
\end{equation}
as needed.
\end{proof}

\section{Proof of Theorem~\ref{cycle indep thm}}
In the course of the proof we shall use the following explicit
Chebotarev theorem which is a special case of
Theorem~\ref{thm:Cheb-gen} of Appendix~\ref{apendix}:

\begin{theorem}\label{thm:cheb}
Let $\bfA=(A_1, \ldots, A_n)$ be an $n$-tuple of variables over
$\FF_q$, let $\F(t) \in \FF_q[\bfA][t]$ be monic, separable, and of
degree $m$ viewed as a polynomial in $t$, let $L$ be a splitting
field of $\F$ over $K=\FF_q(\bfA)$, and let
$G=\Gal(\F,K)=\Gal(L/K)$. Assume that $\FF_q$ is algebraically
closed in $L$. Then there exists a constant $c=c(n,{\rm tot.deg}
(\F))$ such that for every conjugacy class $C\subseteq G$ we have
\[
\left|\# \{ \bfa\in \FF_q^n : \Fr_\bfa = C\} -\frac{|C|}{|G|}
q^n\right|\leq c q^{n-1/2}.
\]

\end{theorem}

Here $\Fr_{\bfa}$ denotes the Frobenius conjugacy class
$\left(\frac{S/R}{\phi}\right)$ in $G$ associated to the
homomorphism  $\phi \colon R\to \FF_q$ given by $\bfA\mapsto \bfa\in
\FF_q^n$, where $R=\FF_q[\bfA,\disc\F^{-1}]$ and $S$ is the integral
closure of $R$ in the splitting field of $\F$. See
Appendix~\ref{apendix}, in  particular \eqref{frob cc}, for more
details.

Let $\bfA=(A_1, \ldots, A_n)$ be an $n$-tuple of variables and set
\begin{equation}
\F_i = T^n + A_1  T^{n-1} + \cdots + A_n + h_i(T) \quad \mbox{and}
\quad \F = \F_1\cdots \F_s,
\end{equation}
where the $h_{i}$'s are distinct polynomials.
Let $L$ be the splitting field of $\F$ over $K=\FF_q(\bfA)$ and let
$\FF$ be an algebraic closure of $\FF_q$. By \cite[Proposition
3.1]{BS},
\[
G:=\Gal(\F,K) = \Gal(L/K) = \Gal(\FF L / \FF K) = S_n^s.
\]
In \cite{BS} it is assumed that $q$ is odd, but using \cite{Carmon}
that restriction can now be removed for $n>2$. This in particular
implies that $L\cap \FF = \FF_q$ (since the image of the restriction
map $\Gal(\FF L/\FF K) \to \Gal(L/K) $ is $\Gal(L/L\cap \FF K)$, so
by the above and Galois correspondence $L\cap (\FF K) = K$, and in
particular $L\cap \FF=K\cap \FF=\FF_q$). Hence we may apply
Theorem~\ref{thm:cheb} with the conjugacy class
\[
C=\{ (\sigma_1, \ldots, \sigma_s)\in G : \lambda_{\sigma_i} =
\lambda_i\}
\]
to get that
\[
|\#\{\bfa\in \FF_{q}^n : \Fr_{\bfa} = C \} - |C|/|G| \cdot q^n|\leq
c(s,n) q^{n-1/2}.
\]
Since $|C|/|G| = p(\lambda_1) \cdots p(\lambda_s)$ and since
$\#\{\bfa\in \FF_{q}^n : \disc_T(\F)(\bfa)=0\} = O_{s,n}(q^{n-1})$,
it remains to show that for $\bfa\in \FF_q^n$ with
$\disc_T(\F(\bfa))\neq 0$ we have $\Fr_{\bfa} =C$ if and only if
$\lambda_{\F_i(\bfa,T)}=\lambda_i$ for all $i=1, \ldots, s$.

And indeed, extend the specialization $\bfA\mapsto \bfa$ to a
homomorphism $\Phi$ of $\FF_q[\bfA,\bfY]$ to $\FF$, where
$\bfY=(Y_{ij})$, and $Y_{i1}, \ldots, Y_{in}$ are the roots of
$\F_i$. Then $\Fr_{\bfa}$ is, by definition, the conjugacy class of
the Frobenius element $\Fr_{\Phi}\in G$ which is defined by
\begin{equation}\label{action of frob}
\Phi(\Fr_{\Phi}(Y_{ij}))=\Phi(Y_{ij})^q.
\end{equation}
Note that $\Fr_\Phi$ permutes the roots of each $\mathcal F_i$ and
hence can be identified with a $s$-tuple of permutations $\Fr_\phi =
(\sigma_1,\ldots, \sigma_s)\in G=S_n^s$. Since the $\Phi(Y_{ij})$
are distinct,  the cycle structure of $\sigma_i$ equals the cycle
structure of the $\Phi(Y_{ij})\to \Phi(Y_{ij})^q$, $j=1,\ldots, n$
by \eqref{action of frob} which in turn equals the cycle structure
of the polynomial $\F_i(\bfa,T)$.  Hence $\Fr_{\Phi}\in C$ if and
only if $\lambda_{\F_i(\bfa,T)}=\lambda_i$ for all $i$, as needed.
\qed


\section{Proof of Theorem~\ref{main thm}}\label{sec:prfmainthm}

First we need the following lemma:

\begin{lemma}
Let $f\in\mathcal{M}_{n}$ and $h\in\mathbb{F}_{q}[t]$ such that $\text{deg}(h)<n$. Then we have that 
\begin{equation}
\#\left\{\text{$f\in\mathcal{M}_{n}$:$f$ and $f+h$ are square-free}\right\}=q^{n}+O(q^{n-1}).
\end{equation}
\end{lemma}
\begin{proof}
The number of square-free $f\in \EM_n$ is $q^n-q^{n-1}$ for
$n\geq 2$ (for $n=1$ it is $q$), and since $n>\deg(h)$, as $f$ runs
over all monic polynomials of degree $n$ so does $f+h$, and hence
the number of $f\in \EM_n$ such that $f+h$ is square-free is also
$q^n-q^{n-1}$. Therefore there are at most $2q^{n-1}$ monic $f\in
\EM_n$ for which at least one of $f$, $f+h$ is not square-free, as
claimed.
\end{proof}

We denote by $\ave{A}$ the mean value of an arithmetic function $A$
over $\EM_n$:
\begin{equation}
\ave{A}:=\frac 1{q^n} \sum_{f\in \EM_n} A(f) \;.
\end{equation}

For this it follows that if $A$ is an arithmetic function on $\EM_n$
that is bounded independently of $q$, then
\begin{equation}\label{eq:squarefreeneglection}
\langle A \rangle = \frac{1}{q^n} \sum_{\substack{f\in \EM_n\\ f \;
\text{and} \; f+h \; \text{square-free} }}A(f) +O\big(q^{n-1}\big).
\end{equation}

Now for square-free $f$, the divisor function $\divid_k(f)$ depends
only on the cycle structure of $f$, namely
\begin{equation}\label{eq:div-cyc}
\divid_k(f) = k^{|\lambda(f)|},
\end{equation}
where for a partition $\lambda=(\lambda_1,\dots, \lambda_n)$ of $n$,
we denote by $|\lambda| = \sum \lambda_j$ the number of part of
$\lambda$. Therefore we may apply
\eqref{eq:squarefreeneglection} with \eqref{eq:div-cyc} to get
\begin{equation}\label{eq:divid_ave}
\left\langle d_{k}(\bullet) d_{k}(\bullet+h)\right \rangle =
\left\langle k^{|\lambda(\bullet)|} k^{|\lambda(\bullet+h)|} \right
\rangle + O(q^{-1}).
\end{equation}
Since the function $k^{\lambda(f)}$  depends only on the cycle
structure of $f$, it follows from Theorem~\ref{cycle indep thm} that
\begin{equation}\label{eq:divid_indep}
\left\langle k^{|\lambda(\bullet)|} k^{|\lambda(\bullet+h)|} \right
\rangle = \left\langle k^{|\lambda(\bullet)|} \right \rangle
\left\langle
 k^{|\lambda(\bullet+h)|}
\right \rangle +O(q^{-1/2}) = \left\langle k^{|\lambda(\bullet)|}
\right \rangle^{2} + O(q^{-1/2}).
\end{equation}
Applying again  \eqref{eq:squarefreeneglection} with
\eqref{eq:div-cyc}  together with Lemma~\ref{lem:mean value of tau}
we conclude that
\begin{equation}\label{eq:mv-sfdivid}
\left\langle k^{|\lambda(\bullet)|} \right \rangle = \left\langle
 \divid_k(\bullet)  \right \rangle + O(q^{-1}) = \binom{n+k-1}{k-1}
+O(q^{-1}) .
\end{equation}
Combining \eqref{eq:divid_ave}, \eqref{eq:divid_indep}, and
\eqref{eq:mv-sfdivid} then gives the desired result. \qed

\section{Proof of Theorem \ref{thm:general case}}
We argue as in Section~\ref{sec:prfmainthm}:
\begin{align*}
\left< \prod_{i=1}^s d_{k_i}(\bullet + h_i) \right> &= \left< \prod_{i=1}^s k_i^{|\lambda(\bullet+h_i)|}\right> + O(q^{-1})\\
        &=\prod_{i=1}^s \left<k_i^{|\lambda_i(\bullet)|} \right> + O(q^{-1/2})\\
        &=\prod_{i=1}^s \binom{n+k_i-1}{k_i-1} + O(q^{-1/2}).
\end{align*}
(Here the first passage  uses \eqref{eq:squarefreeneglection} with
\eqref{eq:div-cyc}, the last also uses Lemma~\ref{lem:mean value of
tau}, and the middle passage is done by invoking Theorem~\ref{cycle
indep thm}.)\qed


\section{Proof of Theorem \ref{TitDiv for P}}
Let $\chp$ be the characteristic function of the primes of degree
$n$, i.e.\
\begin{equation}
\chp(f) = \chi_{(0,0,\ldots,0,1)}(f) = \left\{
    \begin{array}{ll}
        1,  & \ \mbox{if } \ f\in \mathcal{P}_n \\
        0, & \ \mbox{otherwise}.
    \end{array}
\right.
\end{equation}
The Prime Polynomial Theorem gives that
$\left<\chp\right>=1/n+O(q^{-1})$ and we have calculated in
Section~\ref{sec:prfmainthm} that
$\left<k^{|\lambda(\bullet)|}\right>=\binom{n+k-1}{k-1}+O(q^{-1})$.
Since these two functions clearly  depend only on cycle structures (recall that $\alpha\neq0$),
Theorem~\ref{cycle indep thm} gives
\begin{equation}
\left< \chp(\bullet)\cdot k^{|\lambda(\bullet)|}\right> =
\left<\chp(\bullet)\right>\left<k^{|\lambda(\bullet)|}\right> =
\frac{1}{n}\binom{n+k-1}{k-1} + O(q^{-1/2}).
\end{equation}
Therefore,
\begin{align*}
\frac{n}{q^n}\sum_{P\in \mathcal{P}_n} d_k(P+\alpha) &= n \left< \chp(\bullet)\cdot k^{|\lambda(\bullet)|} \right>\\
        & = \binom{n+k-1}{k-1} + O(q^{-1/2}),
\end{align*}
as needed. \qed


\section{Comparing conjectures and our results}
\label{sec:comparison}

In this section we check the compatibility of the theorems presented
in Section \S~\ref{sec: additive divisor problem} with the known
results over the integers.

\subsection{Estermann's theorem for $\fq[t]$}
First we prove the function field analogue of Estermann's result
\eqref{eq:Estermann}. For simplicity, we carry it out for  $h=1$.

\begin{theorem}\label{function field Estermann}
Assume that $n\geq1$. Then

\begin{equation}
\label{eq: ff Estermann}
\frac{1}{q^{n}}\sum_{f\in\mathcal{M}_{n}}d_{2}(f)d_{2}(f+1)=(n+1)^{2}-\frac{1}{q}(n-1)^{2}.
\end{equation}
(Note that $q$ is fixed in this theorem).
\end{theorem}

We need two auxiliary lemmas before proving Theorem~\ref{function
field Estermann}.

Let $A,B\in\mathbb{F}_{q}[t]$ be monic polynomials. We want
to count the number of monic polynomials solutions
$(u,v)\in\mathbb{F}_{q}[t]^{2}$ of the linear Diophantine equation

\begin{equation}
\label{eq: Dioph. Equa} Au-Bv=1, \ \ \ \ \
\mathrm{deg}(Au)=n=\mathrm{deg}(Bv)\;.
\end{equation}
As follows from the Euclidean algorithm, a necessary and sufficient
condition for the equation $Au-Bv=1$ to be solvable in
$\mathbb{F}_{q}[t]$ is $\mathrm{gcd}(A,B)=1$.

\begin{lemma}
\label{Dioph. Equation} Given monic polynomials
$A,B\in\mathbb{F}_{q}[t]$, $\mathrm{gcd}(A,B)=1$, and

\begin{equation}
\label{eq: Dioph. Equa1} n\geq\mathrm{deg}(A)+\mathrm{deg}(B)
\end{equation}
then the set of monic solutions $(u,v)$ of \eqref{eq: Dioph. Equa}
forms a nonempty affine subspace of dimension
$n-\mathrm{deg}(A)-\mathrm{deg}(B)$, hence the number of solutions
is exactly $q^{n}/|A||B|$.
\end{lemma}

\begin{proof}
We first ignore the degree condition. By the theory of the linear
Diophantine equation, given a particular
solution $(u_{0},v_{0})\in\mathbb{F}_{q}[t]^{2}$, all other
solutions in $\mathbb{F}_{q}[t]^{2}$ are of the form

\begin{equation}
\label{eq: Dioph. Equa2} (u_{0},v_{0})+k(B,A)
\end{equation}
where $k\in\mathbb{F}_{q}[t]$ runs over all polynomials.

Given $u_{0}$, we may replace it by $u_{1}=u_{0}+kB$ where
$\mathrm{deg}(u_{1})<\mathrm{deg}(B)$ (or is zero), so that we may
assume that the particular solution satisfies

\begin{equation}
\label{eq:DE3} \mathrm{deg}(u_{0})<\mathrm{deg}(B).
\end{equation}
In that case, if $k\neq0$ then

\begin{equation}
\label{eq:DE4} \mathrm{deg}(u_{0}+kB)=\mathrm{deg}(kB)
\end{equation}
and $u_{0}+kB$ is monic if and only if $k$ is monic. Hence if $k\neq0$, then

\begin{align}
\label{eq:DE5}
\mathrm{deg}(u_{0}+kB)=n-\mathrm{deg}(A)&\Leftrightarrow\mathrm{deg}(kB)=n-\mathrm{deg}(A)\nonumber \\
&\Leftrightarrow\mathrm{deg}(k)=n-\mathrm{deg}(A)-\mathrm{deg}(B).
\end{align}

Thus the set of solutions of \eqref{eq: Dioph. Equa} is in
one-to-one correspondence with the space
$\mathcal{M}_{n-\mathrm{deg}(A)-\mathrm{deg}(B)}$ of monic $k$ of
degree $n-\mathrm{deg}(A)-\mathrm{deg}(B)$. In particular the number
of solutions is $q^{n}/|A||B|$.
\end{proof}

Let 
\begin{equation}
\label{eq: S-lemma}
S(\alpha,\beta;\gamma,\delta):=\#\left\{x\in\mathcal{M}_{\alpha},
y\in\mathcal{M}_{\beta},z\in\mathcal{M}_{\gamma},
u\in\mathcal{M}_{\delta}:xy-zu=1\right\}.
\end{equation}
Then we have the following lemma.

\begin{lemma}
\label{S-lemma} For $\alpha+\beta=n=\gamma+\delta$,
\begin{equation}
\label{eq: S-lemma1} S(\alpha,\beta;\gamma,\delta) = q^{n} \times
\begin{cases} 1, & \mbox{if } \mathrm{min}(\alpha,\beta;\gamma,\delta)=0, \\
1-\frac{1}{q}, & \mbox{otherwise}.
\end{cases}
\end{equation}
\end{lemma}

\begin{proof}
We have some obvious symmetries from the definition
\begin{equation}
S(\alpha,\beta;\gamma,\delta) = S(\beta,\alpha;\gamma,\delta) =
S(\alpha,\beta;\delta,\gamma)
\end{equation}
and hence to evaluate $S(\alpha,\beta;\gamma,\delta)$ it suffices to
assume
\begin{equation}\label{ordering}
\alpha\leq \beta,\quad \gamma\leq \delta.
\end{equation}
Assuming \eqref{ordering}, we write

\begin{equation}
S(\alpha,\beta;\gamma,\delta) = \sum_{\substack{x\in
\EM_\alpha\\z\in \EM_\gamma\\\gcd(x,z)=1}} \#\{y\in \EM_\beta,u\in
\EM_\delta: xy-zu=1\}
\end{equation}
Note that $\alpha,\gamma\leq n/2$ (since $\alpha+\beta=n$ and
$\alpha\leq \beta$) and hence $\alpha+\gamma\leq\tfrac{1}{2}(\alpha+\beta+\gamma+\delta)=n$. Thus we may use
Lemma~\ref{Dioph. Equation} to deduce that

\begin{equation}
\#\{y\in \EM_\beta,u\in \EM_\delta: xy-zu=1\} = q^{n-\alpha-\gamma}
\end{equation}
and therefore
\begin{equation}
S(\alpha,\beta;\gamma,\delta) =
q^{n-\alpha-\gamma}\sum_{\substack{x\in \EM_\alpha\\z\in
\EM_\gamma\\\gcd(x,z)=1}}1.
\end{equation}
Recall the M\"obius inversion formula, which says that for monic $f$, 
$\sum_{d\mid f}\mu(d)$ equals $1$ if $f=1$, and $0$ otherwise. 
Hence we may write the coprimality condition $\gcd(x,z)=1$ using the
M\"{o}bius function as
\begin{equation}
\sum_{d\mid x, \ d\mid z} \mu(d) = \begin{cases} 1,&
\gcd(x,z)=1,\\0,&{\rm otherwise}. \end{cases}
\end{equation}
and therefore
\begin{equation}
\label{eq11}
\begin{split}
S(\alpha,\beta;\gamma,\delta) &=
q^{n-\alpha-\gamma}\sum_{\substack{x\in \EM_\alpha\\z\in \EM_\gamma
}}\sum_{d\mid x, d\mid z} \mu(d)  \\
\\&=q^{n-\alpha-\gamma}\sum_{\substack{\deg (d)\leq \min(\alpha,\gamma)\\d\;{\rm monic}}}
\mu(d) \#\{x\in \EM_\alpha: d\mid x\}\cdot \#\{z\in \EM_\gamma
:d\mid z\} \\
&=q^{n-\alpha-\gamma}\sum_{\substack{\deg (d)\leq
\min(\alpha,\gamma)\\d\;{\rm monic}}} \mu(d)\frac{q^\alpha}{|d|}
\cdot \frac{q^\gamma}{|d|}\\
&=q^n\sum_{\substack{\deg (d)\leq \min(\alpha,\gamma)\\d\;{\rm
monic}}}\frac{ \mu(d)}{|d|^2}\\
&=q^n\sum_{\substack{\deg (d)\leq \min(\alpha,\beta;\gamma,\delta)\\d\;{\rm
monic}}}\frac{ \mu(d)}{|d|^2},
\end{split}
\end{equation}
where we have used the fact that $\alpha\leq \beta$ and $\gamma\leq \delta$.

We next claim that 
\begin{equation}\label{eq12}
\sum_{\substack{\deg (d)\leq \eta \\d\;{\rm monic}}}\frac{
\mu(d)}{|d|^2} = \begin{cases} 1,& \eta=0,\\1-\frac 1q,& \eta\geq 1,
\end{cases}
\end{equation}
which when we insert into \eqref{eq11} proves the lemma.

To prove \eqref{eq12}, we sum over $d$ of fixed degree
\begin{equation}
\sum_{\substack{\deg (d)\leq \eta \\d\;{\rm monic}}}\frac{
\mu(d)}{|d|^2} =\sum_{0\leq \xi\leq \eta}\frac 1{q^{2\xi}}
\sum_{d\in \EM_\xi} \mu(d)
\end{equation}
and recall that (\cite[Chapter 2 - Exercise 12]{Rosen})
\begin{equation}
\sum_{d\in \EM_\xi} \mu(d) =
\begin{cases}1,&\xi=0\\-q,&\xi=1\\0,&\xi\geq 2
\end{cases}
\end{equation}
 from which \eqref{eq12} follows.
\end{proof}

\begin{proof}[Proof of Theorem \ref{function field Estermann}]
We write
\begin{equation}
\begin{split}
\nu&:= \sum_{f\in \EM_n} \divid_{2}(f)\divid_{2}(f+1) \\
& = \#\{x,y,z,u\in \fq[t]\;{\rm monic}: xy-zu=1,\quad
\deg(xy)=n=\deg(zu)\}.
\end{split}
\end{equation}
We partition this into a sum over variables with fixed degree, that
is
\begin{equation}\label{eq:part}
\nu = \sum_{\substack{\alpha+\beta=n\\\gamma+\delta=n\\
\alpha,\beta,\gamma,\delta\geq 0}} S(\alpha,\beta;\gamma,\delta).
\end{equation}

We now input the results of Lemma~\ref{S-lemma} into \eqref{eq:part}
to deduce that
\begin{equation}
\nu =\sum_{\substack{\alpha+\beta=n\\\gamma+\delta=n\\
\alpha,\beta,\gamma,\delta\geq 0}} q^n \times
\begin{cases} 1,&\min(\alpha,\beta;\gamma,\delta)=0,\\
1-\frac 1q,& {\rm otherwise}.
\end{cases}
\end{equation}
Of the $(n+1)^2$ quadruples of non-negative integers
$(\alpha,\beta;\gamma,\delta)$ so that
$\alpha+\beta=n=\gamma+\delta$,   there are exactly $4n$ tuples
$(\alpha,\beta;\gamma,\delta)$ for which
$\min(\alpha,\beta)=0=\min(\gamma,\delta)$, namely they are

\begin{equation}
(n,0;n,0),\quad (n,0;0,n),\quad (0,n;n,0),\quad (0,n;0,n)
\end{equation}
and the $4(n-1)$ tuples of the form

\begin{equation}
(n,0;i,n-i), (0,n;i,n-i), (i,n-i;n,0),(i,n-i;0,n)
\end{equation}
for $0<i<n$.

Concluding, we have

\begin{equation}
\begin{split}
\nu &= (4+4(n-1))\cdot q^n  + \left[(n+1)^{2}-(4+4(n-1))\right]\cdot q^n\left(1-\frac 1q\right)\\
&=q^n\left( (n+1)^2 - \frac 1q (n-1)^{2} \right)
\end{split}
\end{equation}
proving the theorem.
\end{proof}

It is easy to check that Theorem \ref{main thm} is compatible with
the function field analogue of Estermann's result. Taking
$q\rightarrow\infty$ in \eqref{eq: ff Estermann} we recover the same
results as presented in \eqref{eq:main thm} with $k=2$.

\subsection{Higher divisor functions}\label{Sec:compatibility}

Next, we want to check compatibility of our result in
Theorem~\ref{main thm} with what is conjectured over the integers.
 It is conjectured that
\begin{equation}
D_k(x;h)\sim xP_{2(k-1)}(\log x;h) \ \ \ \ \ \ \ \ \text{as
$x\rightarrow\infty$},
\end{equation}
where  $P_{2(k-1)}(u;h)$ is a polynomial in $u$ of degree $2(k-1)$,
whose coefficients depend on $h$ (and $k$). This conjecture appears 
in the work of Ivi\'{c} \cite{Ivic97}, and Conrey and Gonek \cite{CG}, 
and from their work, with some effort, we can explicitly write the 
conjectural leading coefficient for the desired polynomial. 
The conjecture over $\mathbb{Z}$ states that 

\begin{equation}\label{conj for P}
P_{2(k-1)}(u;h)= \frac{1}{[(k-1)!]^{2}} A_k(h) u^{2k-2}+\dots ,
\end{equation}
 where
\begin{equation}\label{formula for Ar}
A_k(h) = \sum_{m=1}^{\infty}\frac{c_{m}(h)}{m^{2}}C_{-k}^{2}(m)
\end{equation}
with
\begin{equation}
\label{expsum}
C_{-k}(m)=m^{1-k}\sum_{a_{1}=1}^{m}\cdots\sum_{a_{k}=1}^{m}e\left(\frac{ha_{1}\cdots
a_{k}}{m}\right),
\end{equation}
where $e(x)=e^{2\pi ix}$ and $c_{m}(h)$ is the Ramanujan sum
\begin{equation}
c_{m}(h)=\sum_{\substack{a=1 \\ (a,m)=1}}^{m}e^{2\pi i\tfrac{a}{m}h}
=\sum_{d\mid \text{gcd}(m,h)} d\mu\left(\frac{m}{d}\right).
\end{equation}

We now translate the conjecture above
to the function field setting using the correspondence
$x\leftrightarrow q^{n}$, $\log x\leftrightarrow n$ and that sum
over positive integers correspond to sum over 
monic polynomials in $\mathbb{F}_{q}[t]$. Under this
correspondence the function field analogue
of the above polynomial is given in the following 
conjecture

\begin{conj}
\label{conj1} For $q$ fixed, let $0\neq h\in \fq[t] $. Then as $n\to
\infty$,
\begin{equation}
\label{eq:conj1} \sum_{f\in \EM_n} \divid_k(f)\divid_k(f+h)\sim
\frac{1}{[(k-1)!]^{2}}A_{k,q}(h)q^{n}n^{2k-2},
\end{equation}
 where
\begin{equation}
\label{eq:conj2}
A_{k,q}(h)=\sum_{\substack{m\in\mathbb{F}_{q}[t] \\ \text{\em monic}}}\frac{c_{m,q}(h)(\gcd(m,h))^{2(k-1)}}{|m|^{2(k-1)}}g_{k-1}^{2}\left(\frac{m}{\gcd(m,h)}\right),
\end{equation}
where $|m|=q^{\mathrm{deg}(m)}$,
\begin{equation}
g_{k-1}(f)=\#\left\{a_{1},\ldots,a_{k-1}\bmod f:a_{1}\ldots
a_{k-1}\equiv0\bmod f\right\},
\end{equation}
and
\begin{equation}
\label{ramanujan sum props over fq}
c_{m,q}(h)=\sum_{d\mid\gcd(m,h)}|d|\mu\left(\frac{m}{d}\right)
\end{equation}
is the Ramanujan sum over $\mathbb{F}_{q}[t]$.
The sum above is over all monic polynomials $d\in\mathbb{F}_{q}[t]$
and $\mu(f)$ is the M\"{o}bius function for $\mathbb{F}_{q}[t]$ and
$\Phi(m)$ is the $\mathbb{F}_{q}[t]$-analogue for Euler's totient function.
\end{conj}
\begin{remark}
Note that 
\begin{equation}
C_{q,-k}^{2}(m)=\frac{\text{gcd}(m,h)^{2k-1}}{|m|^{k-1}}g_{k-1}^{2}\left(\frac{m}{\text{gcd}(m,h)}\right)
\end{equation}
correspond to $C_{-k}^{2}(m)$ as given in \eqref{expsum}.
\end{remark}
\begin{remark}
Note that we establish this conjecture for $k=2$ and $h=1$ in
Theorem~\ref{function field Estermann}.
\end{remark}

We now check that our Theorem~\ref{main thm}  
is consistent  with the
conjecture~\eqref{conj for P} and \ref{eq:conj2} for the leading term of the
polynomial $P_{2(k-1)}(u;h)$.

The polynomial given by Theorem \ref{main thm} is
\begin{equation}
\binom{n+k-1}{k-1}^2 = \frac 1{[(k-1)!]^2} n^{2(k-1)}+\dots.
\end{equation}

We wish to show that as $q\to \infty$, $A_{k,q}(h)/[(k-1)!]^2$
matches the leading coefficient of $\binom{n+k-1}{k-1}^2$, that is
\begin{equation}\label{matching Ar}
\lim_{q\to \infty} A_{k,q}(h) =1.
\end{equation}

Indeed, from \eqref{ramanujan sum props over fq} we note that
$|c_{m,q}(h)| = O_h(1)$, and it is easy to see that
\begin{equation}
g_{k-1}(n)\leq n^{k-1}\divid(n)^{k-1} \ll |n|^{k-2+\epsilon}, \quad
\forall \epsilon>0.
\end{equation}
Thus we find
\begin{equation}
A_{k,q}(h) =1+   O\Big(\sum_{\substack{m\in \mathcal M\\ \deg(m)
>0}} \frac 1{|m|^{2-\epsilon}}\Big).
\end{equation}
The series in the $O$-term is a geometric series:
\begin{equation}
\sum_{\substack{m\in \mathcal M\\ \deg(m)
>0}} \frac 1{|m|^{2-\epsilon}} = \sum_{n=1}^\infty \frac
1{q^{n(2-\epsilon)}} \#\mathcal M_n =\sum_{n=1}^\infty \frac
1{q^{n(1-\epsilon)}}= \frac{1/q^{1-\epsilon}}{1- 1/q^{1-\epsilon}}
\end{equation}
and hence tends to $0$ as $q\to \infty$, giving \eqref{matching Ar}.
\vspace{1.0cm}

\noindent\textbf{Acknowledgments.} We thank an anonymous referee for 
detailed comments and suggestions.



\appendix
\section{An explicit Chebotarev theorem}\label{apendix}
We prove an explicit Chebotarev theorem for function fields over
finite fields. This theorem is known to experts, 
cf. \cite[Theorem 4.1]{Cha}, \cite[Proposition 6.4.8]{FJ} or
\cite[Theorem 9.7.10]{KS}, however there it is not given explicitly 
with the uniformity that we need to use. Therefore we provide a complete proof.
\subsection{Frobenius elements}
Let $\FF_q$ be a finite field with $q$ elements and algebraic
closure $\FF$. We denote by $\Fr_q$ the Frobenius automorphism
$x\mapsto x^q$.

Let $R$ be an integrally closed finitely generated $\FF_q$-algebra
with fraction field $K$,  let $\F\in R[T]$ be a monic separable
polynomial of degree $\deg \F=m$ such that
\begin{equation}\label{eq:disc_inv}
\disc\F \in R^*
\end{equation}
is invertible. Let $\bfY=(Y_1, \ldots, Y_m)$ be the roots of $\F$,
and  put
\[
S=R[\bfY],\qquad L=K(\bfY),\qquad \mbox{and} \qquad G=\Gal(L/K).
\]
We identify $G$ with a subgroup of $S_m$ via the action on $Y_1,
\ldots, Y_m$:
\begin{equation}
\label{eq:embedding_into_Sm} g(Y_{i}) = Y_{g(i)}, \qquad g\in G\leq
S_m.
\end{equation}
By \eqref{eq:disc_inv} and Cramer's rule, $S$ is the integral
closure of $R$ in $L$ and $S/R$ is unramified. In particular, the
relative algebraic closure $\FF_{q^{\mu}}$ of $\FF_q$ in $L$ is
contained in $S$. For each $\nu \geq 0$ we let
\begin{equation}\label{eqdef:G_nu}
G_\nu = \{ g\in G : g(x) = x^{q^\nu}, \ \forall x\in \FF_{q^\mu}\},
\end{equation}
the preimage of $\Fr_q^\nu$ in $G$ under the restriction map. Since
$\Gal(\FF_{q^\nu}/\FF_q)$ is commutative, $G_\nu$ is stable under
conjugation.

For every $\Phi\in \Hom_{\FF_q}(S,\FF)$ with $\Phi(R)=\FF_{q^\nu}$
there exists a unique element in $G$, which we call the
\emph{Frobenius element} and denote by
\begin{equation}\label{eq:not_Frob}
\left[\frac{S/R}{\Phi}\right] \in G,
\end{equation}
such that
\begin{equation}\label{eq:Frob}
\Phi\left(\left[\frac{S/R}{\Phi}\right] x\right) =
\Phi(x)^{q^{\nu}}, \qquad \forall x\in S.
\end{equation}
Since $S$ is generated by $\bfY$ over $R$, it suffices to consider
$x\in\{Y_1, \ldots, Y_k\}$ in \eqref{eq:Frob}. If we further assume
that $\Phi\in \Hom_{\FF_{q^\mu}}(S,\FF)$, then \eqref{eq:Frob} gives
that $\left[\frac{S/R}{\Phi}\right]x=x^{q^{\nu}}$ for all $x\in
\FF_{q^\mu}$, hence

\begin{equation}\label{eq:inGnu}
\Phi(R)=\FF_{q^\nu} \Longrightarrow \left[\frac{S/R}{\Phi}\right]\in
G_\nu.
\end{equation}

\begin{lemma}\label{lem:matrix}
For every $g\in S_m$ and $\nu\geq 1$ there exists
$V_{g,\nu}=(v_{ij})\in \GL_m(\FF)$ such that such that $\Fr_{q^\nu}$
acts on the rows of $V_{g,\nu}$ as $g$ acts on $\bfY$:
\begin{equation}\label{eq:actonmatrix}
v_{ij}^{q^\nu} = v_{g(i)j}.
\end{equation}
\end{lemma}

\begin{proof}
By replacing $q$ by $q^\nu$, we may assume without loss of
generality that $\nu=1$. By relabelling we may assume without loss
of generality that
\begin{equation}\label{eq:permutation}
g =
 (s_1\ \cdots\ e_1)(s_2\ \cdots\ e_2) \cdots (s_{k}\ \cdots \ e_{k}),
\end{equation}
where $s_1=1$, $s_{i+1}=e_i+1$, and $e_k=m$.

Let $V$ be the block diagonal matrix
\[
V=\begin{pmatrix} V_{1}  \\  & V_{2} &  \\ &  & \ddots    \\ & & &
V_{k}\end{pmatrix} ,
\]
where
\[
V_{i}=
\begin{pmatrix}
1& \zeta_i& \cdots &\zeta_i^{\lambda_i-1}\\
1& \zeta_i^q& \cdots &\zeta_i^{q(\lambda_i-1)}\\
\vdots&\vdots && \vdots\\
1& \zeta_i^{q^{\lambda_i-1}}& \cdots
&\zeta_i^{q^{\lambda_i-1}(\lambda_i-1)}
\end{pmatrix},
\]
is the vandermonde matrix corresponding to an element  $\zeta_i\in
\FF$  of degree $\lambda_i=e_i-s_i$ over $\FF_q$. So $\det
V_i=\prod_{1\leq j'<j\leq \lambda_i}
(\zeta_i^{q^{j'-1}}-\zeta_i^{q^{j-1}})\neq 0$, hence $V$ is
invertible, and by definition $\Fr_q$ acts on the rows of $V$ as the
permutation $g$.
\end{proof}

\begin{lemma}
Let $\Phi\colon S \to \FF$ with $\Phi(R)=\FF_{q^\nu}$ and let $g\in
G_{\nu}$. Then
\begin{equation} \label{eq:frob-point}
\left[ \frac{S/R}{\Phi}\right] = g \Longleftrightarrow
V^{-1}\left(\begin{smallmatrix}\Phi(Y_1)\\ \vdots \\ \Phi(Y_m)
\end{smallmatrix}\right) \in \FF_{q^\nu}^m,
\end{equation}
where $V=V_{g,\nu}$ is the matrix from Lemma~\ref{lem:matrix}.
\end{lemma}

\begin{proof}
Let $z_1, \ldots, z_m\in \FF$ be the unique solution of the linear
system
\begin{equation}\label{eq:syseqn}
\Phi(Y_i) = \sum_{j=1}^m v_{ij} z_j, \qquad i=1,\ldots m,
\end{equation}
i.e.\ $\left(\begin{smallmatrix}z_1 \\ \vdots \\ z_m
\end{smallmatrix}\right)=V^{-1}\left(\begin{smallmatrix}\Phi(Y_1)\\ \vdots \\ \Phi(Y_m)
\end{smallmatrix}\right)$.
If $z_i\in \FF_{q^\nu}$, i.e.\ $z_i^{q^\nu}=z_i$, we get by applying
$\Fr_{q^\nu}$ on \eqref{eq:syseqn} that
\[
\Phi(Y_i)^{q^\nu} = \sum_{j=1}^m v_{ij}^{q^{\nu}} z_i = \sum_{j=1}^m
v_{g(i) j}z_i = \Phi(Y_{g(i)}).
\]
Hence $\left[ \frac{S/R}{\Phi}\right] = g$ by \eqref{eq:Frob}.

Conversely, if $\left[ \frac{S/R}{\Phi}\right] = g$, then
$\Phi(Y_{i})^{q^\nu} = \Phi(Y_{g(i)})$ by
\eqref{eq:embedding_into_Sm} and \eqref{eq:Frob}. We thus get that
$\Fr_{q^\nu}$ permutes the equations in \eqref{eq:syseqn}, hence
$\Fr_{q^\nu}$ fixes the unique solution of \eqref{eq:syseqn}. That
is to say, $z_i^{q^\nu}=z_i$, as needed.
\end{proof}

Next we describe the dependence of the Frobenius element when
varying the homomorphisms. For $\phi\in \Hom_{\FF_q}(R,\FF)$ we
define
\begin{equation}\label{frob cc}
\left(\frac{S/R}{\phi}\right) = \left\{\left[\frac{S/R}{\Phi}\right]
: \Phi\in \Hom_{\FF_{q^\mu}}(S,\FF) \mbox{ prolongs $\phi$}\right\}.
\end{equation}
Unlike the case when working with ideals, this set is not a
conjugacy class in $G$, since we fix the action on $\FF_{q^\mu}$.
However as we will prove below, the group $G_0$ acts regularly on
$\left(\frac{S/R}{\phi}\right)$ by conjugation. In particular if
$G_0=G$, or equivalently if $L\cap \FF=\FF_q$ (with $\FF$ denoting
an algebraic closure of $\FF_q$) then
$\left(\frac{S/R}{\phi}\right)$ is a conjugacy class.

To state the result formally, we recall that a group $\Gamma$ acts
\emph{regularly} on a set $\Omega$ if the action is free and
transitive, i.e.\ for every $\omega_1,\omega_2\in \Omega$ there
exists a unique $\gamma\in \Gamma$ with $\gamma \omega_1 =
\omega_2$.

\begin{lemma}\label{lem:fixed-phi}
Let $\phi\in \Hom_{\FF_q}(R,\FF)$ and let $H$ be the subset of
$\Hom_{\FF_{q^\mu}}(S,\FF)$ consisting of all homomorphisms
prolonging $\phi$. Assume that $\phi(R)=\FF_{q^\nu}$.
\begin{enumerate}
\item The group $G_0$ defined in \eqref{eqdef:G_nu} acts regularly on $H$ by $g\colon \Phi\mapsto \Phi\circ g$.
\item for every $g\in G_0$ and $\Phi\in H$ we have
\[
\left[ \frac{S/R}{\Phi\circ g}\right] = g^{-1}\left[
\frac{S/R}{\Phi}\right] g.
\]
\item
Let $\Phi\in H$, let $g=\left[ \frac{S/R}{\Phi}\right]$, $H_g = \{
\Psi\in H : \left[ \frac{S/R}{\Psi}\right]=g\}$, and $C_{G_0}(g)$
the centralizer of $g$ in $G_0$. Then $C_{G_0}(g)$ acts regularly on
$H_g$.
\item $\# H_g = \#G_0/\#C= \#G/\mu\cdot \#C$, where $C$ is the conjugacy class of $g$ in $G_0$.
\end{enumerate}
\end{lemma}

\begin{proof}
We consider $G_0\leq G$ as subgroups of $S_m$ via the action on
$Y_1,\ldots, Y_m$. Let $g\in G_0$ and $\Phi\in H$. Then $g(x) = x$
and $\Phi(x)=x$, thus $\Phi\circ g(x)=x$, for all $x\in
\FF_{q^\mu}$. Thus $\Phi\circ g\in H$. If $\Phi\circ g=\Phi$, then
$\Phi(Y_{g(i)})=\Phi(Y_{i})$ for all $i$. Since $\disc \F\in R^*$ it
follows that $\Phi(\disc\F)\neq 0$, thus $\Phi$ maps $\{Y_1, \ldots,
Y_m\}$ injectively onto $\{\Phi(Y_1), \ldots, \Phi(Y_m)\}$. We thus
get that $Y_{g(i)}=Y_{i}$, hence $g$ is trivial. This proves that
the action is free.

Next we prove that the action is transitive. Let $\Phi,\Psi\in H$.
Then $\ker\Phi$ and $\ker \Psi$ are prime ideals of $S$ that lies
over the prime ideal $\ker \phi$ of $R$; hence over the prime
$\ker\phi \FF_{q^\mu}$ of $R\FF_{q^\mu}$. By \cite[VII,2.1]{Lang},
there exists $g_1\in \Gal(L/K\FF_{q^\mu})=G_0$ such that
$\ker(\Phi\circ g_1^{-1})= g_1\ker \Phi = \ker\Psi$. Replace $\Phi$
by $\Phi\circ g_1^{-1}$ to assume without loss of generality that
$\ker \Phi=\ker\Psi$. Hence $\Phi = \alpha\circ \Psi$, where
$\alpha$ is an automorphism of the image $\Phi(S)=\Psi(S)$ that
fixes both $\FF_{q^\mu}$ and $\phi(R)=\FF_{q^\nu}$. That is to say,
$\alpha=\Fr_{q}^\rho$, where $\rho$ is a common multiple of $\nu$
and $\mu$. By \eqref{eq:Frob}
\[
\Phi(x) = \Psi(x)^{q^{\rho}}= \Psi\left(\left[
\frac{S/R}{\Psi}\right]x\right)^{q^{\rho-\nu}} = \cdots =
\Psi\left(\left[ \frac{S/R}{\Psi}\right]^{\rho/\nu} x\right),
\]
so $\Phi = \Psi\circ g$, where $g=\left[
\frac{S/R}{\Psi}\right]^{\rho/\nu}$. Since, for $x\in \FF_{q^\mu}$
we have $g(x) = x^{q^\rho}$ and $\mu\mid \rho$, we have $g(x)=x$, so
$g\in G_0$. This finishes the proof of (1).

To see (2) note that
\[
\begin{split}
\Phi\left(g\left[ \frac{S/R}{\Phi\circ g}\right]x\right)&=\Phi\circ g\left(\left[ \frac{S/R}{\Phi\circ g}\right]x\right) \\
&=\Phi\circ g (x)^{q^\nu} = \Phi(gx)^{q^\nu} \\
&= \Phi\left(\left[ \frac{S/R}{\Phi}\right] g x\right), \qquad
\mbox{for all $x\in S$,}
\end{split}
\]
so $g\left[ \frac{S/R}{\Phi\circ g}\right]=\left[
\frac{S/R}{\Phi}\right] g$ (since $\Phi$ is unramified), as claimed.

The rest of the proof is immediate as (3) follows immediately from
(1) and (2) and (4) follows from (3).
\end{proof}

By \eqref{eq:inGnu} and Lemma~\ref{lem:fixed-phi} it follows that if
$\Phi(R)=\FF_{q^\nu}$, then $\left(\frac{S/R}{\phi}\right)\subseteq
G_{\nu}$ is an orbit of the action of conjugation from $G_0$.

Let $C\subseteq G$ be such an orbit, i.e.\ $C=C_g=\{hgh^{-1}:h\in
G_0\}$, $g\in G_\nu$. Then $C\subseteq G_\nu$, since the latter is
stable under conjugation (see after \eqref{eqdef:G_nu}). The
explicit Chebotarev theorem gives the asymptotic probability that
$\left(\frac{S/R}{\phi}\right)=C$:
\[
P_{\nu,C} = \frac{\# \left\{ \phi \in \Hom_{\FF_q}(R,\FF) :
\phi(R)=\FF_{q^\nu} \mbox{ and }
\left(\frac{S/R}{\phi}\right)=C\right\}}{\# \{ \phi \in
\Hom_{\FF_q}(R,\FF) : \phi(R)=\FF_{q^\nu}\}}.
\]

\newcommand{\comp}{{\rm cmp}}

\begin{theorem}\label{thm:Cheb-gen}
Let $\nu\geq 1$, let $C\subseteq G_\nu$ be an orbit of the action of
conjugation from $G_0$. Then
\[
P_{\nu,C} = \frac{\#C}{\#G_\nu} + O_{\deg\F,\comp(R)}(q^{-1/2}),
\]
as $q\to \infty$.
\end{theorem}

We define $\comp(R)$ below.

Before proving this theorem, we need to recall the Lang-Weil
estimates which play a crucial role in the proof of the theorem and
in particular give the asymptotic value of the denominator of
$P_{\nu,C}$.

Let $U$ be a closed subvariety of $\mathbb{A}^n_{\FF_q}$ that is
geometrically irreducible. Lang-Weil estimates give that
\begin{equation}\label{LW}
\# U(\FF_q) = q^{\dim U} + O_{n,\deg U}(q^{\dim U- 1/2}).
\end{equation}
Note that both $n$ and $\deg U$ are stable under base change. This
may be reformulated in terms of $\FF_q$-algebras, to say that if
\begin{equation}\label{R:presentation}
R\cong\FF_q[X_1, \ldots, X_n, f_0^{-1}]/(f_1, \ldots, f_k),
\end{equation}
then
\begin{equation}\label{Lang-Weil}
\# \{ \phi \in \Hom_{\FF_q}(R,\FF) : \phi(R)=\FF_{q}\} = q^{\nu \dim
R} + O_{\comp(R)}(q^{\dim R-\frac12}),
\end{equation}
provided $R\otimes \FF$ is a domain, where $\comp(R)$ is a function
of $\sum \deg f_i$ and $n$, taking minimum over all presentations
\eqref{R:presentation}. By the remark following \eqref{LW}, it
follows that if two $\FF_q$-algebras $S$ and $S'$ become isomorphic
over $\FF$, then $\comp(S')$ is bounded in terms of $\comp(S)$. A
final property needed is that if $R\to S$ is a finite map of degree
$d$, then $\comp(S)$ is bounded in terms of $\comp(R)$ and $d$.

\begin{proof}
Let $g\in C$, let $V=V_{g,\nu}$ be as in \eqref{eq:actonmatrix} and
let
 $S' = R[\bfZ]$, where $\bfZ=V^{-1} \bfY$. Note that $\bfZ$ is the unique solution of the linear system
 \begin{equation}\label{eq:DEFZ}
 Y_i = \sum_{j=1}^n v_ij Z_j, \qquad i=1,\ldots, n.
 \end{equation}

 Let $N=\#\Hom_{\FF_q}(S',\FF_{q^\nu})$.
By \eqref{eq:frob-point}, the number of $\Phi\in
\Hom_{\FF_q}(S,\FF)$ with $\left[\frac{S/R}{\Phi}\right]=g$ equals
$N$. By Lemma~\ref{lem:fixed-phi}, for each $\phi$ there exist
exactly $\#G_0/\#C$ homomorphisms $\Phi\in \Hom_{\FF_q}(S,\FF)$ with
$\left[\frac{S/R}{\Phi}\right]=g$ prolonging $\phi$. Hence,
 \[
\# \left\{ \phi \in \Hom_{\FF_q}(R,\FF) : \phi(R)=\FF_{q^\nu} \mbox{
and } \left(\frac{S/R}{\phi}\right)=C\right\} = \#C/\#G_0 \cdot N.
\]
Since $G_\nu$ is a coset of $G_0$, $\#G_0=\#G_\nu$. Hence it
suffices to prove that $N=q^{\nu\dim
R}+O_{\comp(R),\deg\F}(q^{\nu-1/2})$: As $R\to S'$ is a finite map
of degree $\deg \F$, we get that $\dim R=\dim S'$ and $\comp(S')$ is
bounded in terms of $\comp(R)$ and $\deg \F$. It suffices to show
that $S'\cap \FF\subseteq \FF_{q^{\nu}}$ since then by
\eqref{Lang-Weil} we have
\[
N= q^{\nu \dim S'}+O_{\comp S'}(q^{\nu \dim S'-1/2}) = q^{\nu \dim
R} +O_{\comp(R),\deg F}(q^{\nu \dim R-1/2}),
\]
and the proof is done.

Let $L$ be the fraction field of $S$ and $K$ of $R$. Since $L/K$ is
Galois and $L\cap \FF=\FF_{q^\mu}$ and since the actions of
$\Fr_{q^\nu}$ and $g$ agrees on $\FF_{q^\mu}$, it follows that there
exists an automorphism $\tau$ of $L\FF$ such that $\tau|_{L} = g$
and $\tau|_{\FF}=\Fr_{q^\nu}$. By \eqref{eq:actonmatrix}  $\tau$
permutes the equations \eqref{eq:DEFZ}, hence fixes $\bfZ$ and thus
$S'$. In particular, if $x\in S'\cap \FF$, then $x^{q^\nu} =
\tau(x)=x$, so $x\in \FF_{q^\nu}$, as was needed to complete the
proof.

\end{proof}

\end{document}